\documentclass{math}
\usepackage{hyperref}
\hypersetup{
colorlinks=true,
urlcolor=blue,
citecolor=blue}
\usepackage[all]{xy,xypic}
\usepackage{amsfonts,amssymb,amsmath,amsgen,amsopn,amsbsy,theorem,graphicx,epsfig}
\usepackage{eufrak,amscd,bezier,latexsym,mathrsfs,enumerate,multirow}
\usepackage[utf8]{inputenc}\usepackage[english]{babel}
\usepackage[dvipsnames]{xcolor}

\yil{}
\vol{}
\fpage{}
\lpage{}
\doi{}

\title{Convergence of a linearly regularized nonlinear wave equation to the $p$-system}

\author[ERBAY et al.]{
\textbf{Hüsnü Ata ERBAY$^{1}$, Saadet ERBAY$^{1}$, Albert Kohen ERKİP$^{2}$\thanks{Correspondence: albert@sabanciuniv.edu}}\\
$^{1}$Department of Natural and Mathematical Sciences, Faculty of Engineering, Ozyegin University, \.{I}stanbul, Turkey, \\ ORCID iD: https://orcid.org/0000-0002-5167-609X\\
$^{1}$Department of Natural and Mathematical Sciences, Faculty of Engineering, Ozyegin University, \.{I}stanbul, Turkey, \\ ORCID iD: https://orcid.org/0000-0002-6080-4591\\
$^{2}$Faculty of Engineering and Natural Sciences, Sabanci University, \.{I}stanbul, Turkey, \\
ORCID iD: https://orcid.org/0000-0001-6353-9386

\\ [1.8em]

\rec{.201}
\acc{.201}
\finv{..201}
}

\amssayisi{2010 {\itshape AMS Mathematics Subject Classification:} 35Q74,   74B20, 74H20, 74J30
\newline
\vspace{-15mm}
\begin{center}
		 \raisebox{-17ex}[0ex][0ex]{~~ \raisebox{.5ex}[0ex][0ex]{\footnotesize  This work is licensed under a Creative Commons Attribution 4.0 International License.}}
 		    \end{center}}

\newcommand{\bc}{\begin{center}}
\newcommand{\ec}{\end{center}}

\numberwithin{equation}{section}

\newtheorem{theorem}{Theorem}[section]

\newtheorem{lemma}[theorem]{Lemma}

\newtheorem{remark}[theorem]{Remark}

\renewcommand{\phi}{\varphi}

\begin{document}

\maketitle

\begin{abstract}
We consider  a second-order nonlinear wave equation with a linear convolution term. When the convolution operator is taken as the identity operator, our equation reduces to the classical elasticity equation which can be written as a $p$-system of first-order differential equations. We first establish the local well-posedness of the Cauchy problem. We then investigate the behavior of solutions to the Cauchy problem in the limit as the kernel function of the convolution integral approaches to the Dirac delta function, that is, in the vanishing dispersion limit.  We consider  two different types of the vanishing dispersion limit behaviors for the convolution operator depending on the form of the kernel function. In both cases, we show that the solutions  converge strongly to the corresponding solutions of the classical elasticity equation.

\end{abstract}

\section{Introduction}
\label{Sec:1}

Assuming that  $u=u(x,t)$ is a real-valued function and   $g$ is a sufficiently smooth nonlinear function satisfying $g(0)=g^{\prime }(0)=0$, we consider the nonlocal nonlinear equation
\begin{equation}
    u_{tt}=B u_{xx}+g(u)_{xx},  \label{nw}
\end{equation}%
where $B$ is the convolution operator  in the $x$-variable with the kernel (measure) $\mu$:
\begin{displaymath}
    (Bu)(x)=(\mu \ast u)(x)=\int_{\mathbb{R}} u(x-y)~d\mu(y).
\end{displaymath}%
Throughout the manuscript we assume that $\mu$ is  an even finite Borel measure on $\mathbb{R}$.  Being even implies that the  Fourier transform $\widehat{\mu }(\xi)$  of the measue $\mu(x)$ is real. We further assume that, for some constants $c_{1}$, $c_{2}$,
\begin{equation}
    0<c_{1}\leq  \widehat{\mu }(\xi )\leq c_{2}.  \label{kernelbound}
\end{equation}%
We note that the right-hand side of the inequality above is trivial with  $c_{2} =\vert \mu\vert(\mathbb{R})$. This condition implies that the operator $B$ is   a positive bounded operator on the Sobolev space $H^{s}(\mathbb{R})$ for any $s$.

If the operator $B$ is taken to be the identity operator $I$ (that is, if the kernel is  the Dirac delta measure  $\delta$),  (\ref{nw}) reduces to the classical elasticity equation
\begin{equation}
    u_{tt}=u_{xx}+g(u)_{xx}  \label{classical}
\end{equation}%
written in dimensionless variables.  Equation (\ref{classical}) is hyperbolic whenever $g^{\prime}(u)>-1$.  It  models the non-dispersive propagation of longitudinal waves in an elastic bar of infinite length, where $u$ denotes the strain defined by $u=w_{x}$ where $w(x,t)$ represents axial displacement at position $x$ and  time $t$. On the other hand, the linear dispersion relation $\xi \mapsto \omega^{2}(\xi)=\xi^{2}\widehat{\mu}(\xi)$ of (\ref{nw}) shows the dispersive nature of the solutions. So  the convolution operator $B$  is responsible for dispersion of  wave solutions to (\ref{nw}).

It is worth mentioning that the class (\ref{nw}) of nonlocal nonlinear wave equations covers various models of dispersive wave propagation. A typical example for the measure $\mu$ is $\mu=\delta+\beta$ where $\delta$ is the Dirac measure and $\beta$ is an even $L^{1} (\mathbb{R})$ function such that $0<c_{1}\leq 1+\widehat{\beta}(\xi)$ for some $c_{1}$.  In this case (\ref{nw}) takes the form
\begin{equation}
    u_{tt}=u_{xx}+\beta \ast u_{xx}+g(u)_{xx}  \label{nwplus}
\end{equation}
with the usual convolution operator
\begin{displaymath}
    (\beta \ast u)(x)=\int_{\mathbb{R}} \beta(x-y)u(y)~dy.
\end{displaymath}
 We also note that the class (\ref{nw}) is closely related to the  nonlinearly regularized wave equation
 \begin{equation}
    u_{tt}=\beta \ast \big(u_{xx}+g(u)_{xx}\big)  \label{nwnonlin}
\end{equation}
 considered in  \cite{Duruk2010,Erbay2021} in which the convolution operator  acts on both the linear  and  nonlinear terms.  We stress that the members of the linearly and nonlinearly regularized classes of nonlinear wave equations are totally different from each other. For instance, if the kernel function $\beta$ in (\ref{nwplus}) is taken as the exponential kernel $\beta(x)=\frac{1}{2}e^{-\vert x \vert}$ in which $B=I+(1-D_{x}^{2})^{-1}$, (\ref{nw}) reduces to
\begin{equation}
    u_{tt}- 2u_{xx}-u_{xxtt}+u_{xxxx}=g(u)_{xx}-g(u)_{xxxx}.  \label{ker-exp}
\end{equation}
However, if we take  the same exponential kernel in the nonlinearly regularized class (\ref{nwnonlin}) considered in \cite{Duruk2010,Erbay2021}, we get the improved Boussinesq equation $u_{tt}-u_{xx}-u_{xxtt}=g(u)_{xx}$. To get an another member of the class (\ref{nwplus}), we now consider the triangular kernel defined by $\beta(x)=\frac{1}{h}\big(1-\frac{\vert x\vert}{h}\big)$ for $\vert x\vert \leq h$ and  $\beta(x)=0$ for $\vert x\vert > h$ where $h$ is a positive constant. If the kernel $\beta$ is taken as the triangular kernel, (\ref{nwplus}) reduces to the differential-difference equation
\begin{equation}
    u_{tt}=u_{xx}+\Delta_{h} u+g(u)_{xx},
\end{equation}
 where $\Delta_{h}$ is the second-order central difference operator defined by $(\Delta_{h} u)(x)=\big(u(x+h)-2u(x)+u(x-h)\big)/h^{2}$. However, if $\beta$ in the nonlinearly regularized wave equation (\ref{nwnonlin}) is taken as the triangular kernel, we get the differential-difference equation $u_{tt}=\Delta_{h} (u+g(u))$  in \cite{Duruk2010,Erbay2021}. When this  last equation is written in terms of $w(x,t)$ defined by $u(x,t)=\big(w(x+h,t)-w(x,t)\big)/h$, it becomes the famous Fermi-Pasta-Ulam-Tsingou equation that describes longitudinal vibrations of an infinite chain of identical particles \cite{Fermi1955}. We refer the reader to Section 7 of \cite{Erbay2021} for more details.

In the present work we are concerned with two issues; the local well-posedness of (\ref{nw}) and the convergence of the solutions of (\ref{nw}) to the solutions of (\ref{classical})  as $B$ approaches the identity operator $I$. The second issue is about the vanishing dispersion (vanishing nonlocality) limit of strong solutions to the Cauchy problem for (\ref{nw}).  This issue is inspired by the convergence result  in  \cite{Erbay2021} where the convergence   from a class of nonlinearly regularized wave equations to the classical elasticity equation was established. We extend here the approach developed in \cite{Erbay2021} to the linearly regularized wave equation (\ref{nw}). For this aim we first parameterize (\ref{nw}) by replacing the operator $B$ by the family of convolution operators  $B_{\varepsilon}$.   We then consider two slightly different approaches for the vanishing dispersion limit. In the first approach   the operators are given by $B_{\varepsilon}=(\delta +\varepsilon \beta)\ast$ with a small parameter $\varepsilon$, a fixed $L^{1}$ function $\beta$ and the Dirac measure $\delta$. Obviously, as $\varepsilon \rightarrow 0$, $B_{\varepsilon}$ converges to the identity operator $I$  and we get (\ref{classical}). In the second approach  the operators are  $B_{\varepsilon}=\mu_{\varepsilon}\ast$ with $\mu_{\varepsilon}(x)=\frac{1}{\sqrt{\varepsilon}}\mu(\frac{x}{\sqrt{\varepsilon}})$. When $\int_{\mathbb{R}} \mu~dx=1$, $B_{\varepsilon}$ converges to $I$ as $\varepsilon \rightarrow 0$ and we get again (\ref{classical}). We note that in the second approach we can get the parameterized form of (\ref{nw}) using the  transformation $(x,t,u)\rightarrow (x/\sqrt{\varepsilon}, t/\sqrt{\varepsilon}, u)$ in (\ref{nw}). So the second approach corresponds to the long-wave limit of (\ref{nw}). In both approaches we show that the difference between the corresponding solutions of (\ref{nw}) and  (\ref{classical}) with the same initial data remains small  if the dispersive effect is sufficiently small.

The plan of this paper is as follows. In Section \ref{Sec:2} we prove the local well-posedness of the Cauchy problem for the linear system associated with (\ref{nw}).  In Section \ref{Sec:3} we establish the  local well-posedness of the Cauchy problem for (\ref{nw}).  In Section \ref{Sec:4} we show that, in the vanishing dispersion  limit,   solutions of the Cauchy problem for (\ref{nw}) converge to  the corresponding solution of (\ref{classical}).

Throughout this paper we will follow  the standard notation for function spaces and norms. The Fourier transform $\widehat{u}$ of $u$ is defined by $\widehat{u}(\xi)=\int_{\mathbb{R}}u(x)e^{-i\xi x}dx$. The norm of $u$ in the Lebesgue space $L^{p}(\mathbb{R})$ ($1\leq p\leq \infty $) is represented by $\Vert \cdot\Vert_{L^{p}}$. The notation $H^{s}=H^{s}(\mathbb{R})$ (for $s\in\mathbb{R}$) is used to denote the $L^{2}$-based Sobolev space of order $s$ on $\mathbb{R}$, with the norm $\Vert u\Vert _{H^{s}}=\big( \int_{\mathbb{R}}(1+\xi ^{2})^{s}|\widehat{u}(\xi )|^{2}d\xi \big)^{1/2}$.  $C$ is a generic positive constant. Partial differentiations are denoted by $D_{x}$ etc. For convenience  we also introduce the notations $X^{s}$ and $Y^{s}$ to refer the spaces  defined by
\begin{equation}
    X^{s}=C\big([0, T], H^{s}\big),~~~~Y^{s}=C\big([0, T], H^{s}\big) \cap C^{1}\big([0, T], H^{s-1}\big), \label{XYuzay}
\end{equation}
for fixed $T>0$.  The associated norms of $X^{s}$ and $Y^{s}$ are given by
\begin{equation}
    \Vert u\Vert_{X^{s}}=\sup\limits_{0\leq t\leq T}\Vert u(t)\Vert _{H^{s}}, ~~~~
    \Vert u\Vert _{Y^{s}}
        =\sup\limits_{0\leq t\leq T}\Vert u(t)\Vert_{H^{s}}+\sup\limits_{0\leq t\leq T}\Vert u_{t}(t)\Vert_{H^{s-1}}, \label{XYnorm}
\end{equation}
respectively. Finally, the notations $\Lambda^{s}=(1-D_{x}^{2})^{s/2}$ and $[\Lambda^{s},f] g=\Lambda^{s}(f g)-f\Lambda^{s}g$ are used throughout the remainder of this study.

\section{Local well-posedness for the linear system}
\label{Sec:2}

 To prove our estimates below and in the next sections, we will need the following commutator estimates given in \cite{Kato1988} and, for more general operators $\sigma(D_{x})$, in \cite{Lannes2013} respectively:
\begin{lemma}\label{lem2.1}
    Let  $s\geq 0$. Then for all $f, g$ satisfying $f\in H^{s}$, $D_{x}f\in L^{\infty}$, $g\in H^{s-1}\cap L^{\infty}$,
    \begin{equation*}
        \big\Vert [ \Lambda^{s},f] g\big\Vert_{L^{2}}
        \leq C \big(\Vert D_{x}f\Vert_{L^{\infty}}\Vert g\Vert_{H^{s-1}}+\Vert f\Vert _{H^{s}}\Vert g\Vert_{L^{\infty }}\big).
    \end{equation*}%
    In particular, when $s>3/2$, due to the Sobolev embeddings $H^{s-1}\subset L^{\infty}$, for all $f, g\in H^{s}$
    \begin{equation*}
        \big\Vert [ \Lambda^{s},f] D_{x}g\big\Vert_{L^{2}}\leq C ~\Vert f\Vert_{H^{s}}~\Vert g\Vert_{H^{s}}.
    \end{equation*}%
\end{lemma}
\begin{lemma}\label{lem2.2}
    Let  $t_{0}>1/2$, $r\geq 0$ and $\sigma \in S^{r}$. If $0\leq r\leq t_{0}+1$ and $f\in H^{t_{0}+1}$, then for all $g\in H^{r-1}$, one has
    \begin{equation}
        \big\Vert [ \sigma(D_{x}), f] g\big\Vert_{L^{2}}
        \leq C~ \Vert f_{x}\Vert_{H^{t_{0}}}~\Vert g\Vert_{H^{r-1}}.  \label{LAN-est}
    \end{equation}%
\end{lemma}

The Cauchy problem
\begin{eqnarray}
    && u_{t}=v_{x},\text{ \ \ \ \ \ \ \ \ \ \ \ \ \ \ \ \ \ \ \ \ \ }u(x,0)=u_{0}(x),~~\label{systema} \\
    && v_{t}=Bu_{x}+g^{\prime }(u)u_{x},\text{ \ \ \ \ \ }v(x,0)=v_{0}(x)  \label{systemb}
\end{eqnarray}%
is equivalent to the Cauchy problem defined by (\ref{nw})  and the initial data $u(x,0)=u_{0}(x)$,  $u_{t}(x,0)=(v_{0}(x))_{x}$. We note that if $B$ is taken as the identity operator, (\ref{systema})-(\ref{systemb}) reduces to the Cauchy problem for the  well-known $p$-system:  $u_{t}=v_{x}$,  $v_{t}=u_{x}+(g(u))_{x}$. The $p$-system appears in a number of physical applications, such as to describe the one-dimensional motion of elastic solids or the isentropic gas dynamics  in Lagrangian coordinates (for more on the $p$-system see, for instance, \cite{Lions1994,Young2002}).

We now consider the linear problem
\begin{eqnarray}
    && u_{t}=v_{x},\text{ \ \ \ \ \ \ \ \ \ \ \ \ \ \ \ \ \ \ }u(x,0)=u_{0}(x),~~ \label{linsysa} \\
    && v_{t}=Bu_{x}+w u_{x},\text{ \ \ \ \ \ \ }v(x,0)=v_{0}(x),  \label{linsysb}
\end{eqnarray}
where  $w=w(x,t)$ is a  given fixed function satisfying the following condition
\begin{equation}
    0<d_{1}\leq c_{1}+w(x, t)\leq d_{2}  ~~\text{for all}~(x,t)\in \mathbb{R}\times [0,T]     \label{hyperbolic}
\end{equation}%
for some constants $d_{1}$,  $d_{2}$ and fixed $T>0$ and the lower bound $c_{1}$ for $\widehat{\mu}$ given in (\ref{kernelbound}). Note that the above inequality for $w$ is satisfied whenever $\Vert w(t)\Vert_{L^{\infty}}$ is small enough. Alternatively, another possibility is the case where $w$ is bounded and nonnegative. We also note that  the hyperbolicity of the linearized system (\ref{linsysa})-(\ref{linsysb})  is guaranteed by the conditions (\ref{kernelbound}) and (\ref{hyperbolic}). For the linearized system (\ref{linsysa})-(\ref{linsysb}) we define the $H^{s}$   "energy" functional
\begin{equation}
    \mathcal{E}_{s}^{2}(t)
        =\frac{1}{2}\int_{\mathbb{R}} \bigg( \big(B^{1/2}\Lambda^{s}u(x,t)\big)^{2}+\big(\Lambda^{s}v(x,t)\big)^{2}+ w(x,t)\big(\Lambda^{s}u(x,t)\big)^{2}\bigg) ~dx.  \label{energy}
\end{equation}%
By (\ref{hyperbolic}), $\mathcal{E}_{s}^{2}(t)$ will be equivalent to the norm $\Vert u(t)\Vert_{H^{s}}^{2}+\Vert v(t)\Vert _{H^{s}}^{2}$.   We now prove the existence of the solution to (\ref{linsysa})-(\ref{linsysb}) for both a fixed $w\in Y^{s}$ satisfying (\ref{hyperbolic}) and initial values $u_{0},v_{0}\in H^{s}$. For the existence proof of the linearized system, we follow Taylor's hyperbolic approach \cite{Taylor2011}. In that respect we  consider  Friederichs mollifier $J^{h}$ given by
\begin{equation*}
        J^{h}\varphi(x)=\frac{1}{h}\int_{\mathbb{R}} \eta \big(\frac{x-y}{h}\big)\varphi(y)~dy
\end{equation*}
with some nonnegative $\eta \in C_{0}^{\infty }(\mathbb{R)}$ and $\int_{\mathbb{R}} \eta(x) dx=1$. The following estimate \cite {Mats2017}
\begin{equation}
        \big\Vert J^{h_{1}}\varphi-J^{h_{2}}\varphi\big\Vert_{H^{s-1}}
            \leq C ~\vert h_{1}-h_{2}\vert~ \Vert \varphi\Vert_{H^{s}}     \label{mol-est}
\end{equation}
will be used throughout the rest of the study.  The mollified system is then
\begin{eqnarray}
        u_{t} &=&J^{h}v_{x},~\text{\ \ \ \ \ \ \ \ \ \ \ \ \ \ \ \ \ \ \ \ \ }u(x,0)=u_{0}(x),  \label{molA} \\
        v_{t} &=&BJ^{h}u_{x}+wJ^{h}u_{x},\text{ \ \ \ \ \ }v(x,0)=v_{0}(x).  \label{molB}
\end{eqnarray}
Being an $H^{s}\times H^{s}$-valued linear ODE system, (\ref{molA})-(\ref{molB}) has unique solution $u_{h},v_{h}\in X^{s}$.
\begin{lemma}\label{lem2.3}
    Let $s>3/2$, $u_{0},v_{0}\in H^{s}$ and $w\in Y^{s}$. Suppose that $(u_{h},v_{h})$ satisfy (\ref{molA})-(\ref{molB}) on $[0,T]$. Then the energy $\mathcal{E}_{s}^{2}=\mathcal{E}_{s}^{2}(u_{h},v_{h})$ satisfies the estimate
    \begin{equation}
        \mathcal{E}_{s}^{2}(t)\leq \mathcal{E}_{s}^{2}(0)~e^{C t\Vert w \Vert_{Y^{s}}}  \label{linestimate}
    \end{equation}
    for $t\in [0,T]$.
\end{lemma}
\begin{proof}
Suppressing both $h$ and $t$;
\begin{eqnarray}\label{xx}
    \frac{d}{dt}\mathcal{E}_{s}^{2}(t)
            &=&\int_\mathbb{R} \bigg( \big(B^{1/2}\Lambda^{s}u\big)\big(B^{1/2}\Lambda^{s}u_{t}\big)
                    +\frac{1}{2} w_{t}\big(\Lambda^{s}u\big)^{2}+w\big(\Lambda^{s}u\big)\big(\Lambda^{s}u_{t}\big)
                    + \big(\Lambda^{s}v\big)\big(\Lambda^{s}v_{t}\big) \bigg) ~dx \notag \\
            &=&\frac{1}{2}\int_\mathbb{R} w_{t}\big( \Lambda ^{s}u\big)^{2}~dx
                +\int_\mathbb{R} \big(B^{1/2}\Lambda^{s}u\big)\big(B^{1/2}\Lambda^{s}J^{h}v_{x}\big)~dx
                +\int_\mathbb{R} \big(\Lambda^{s}v\big)\big(\Lambda^{s}BJ^{h}u_{x}\big)~dx
                 \notag \\
            &&  +\int_\mathbb{R} w\big(\Lambda^{s}u\big)\big(\Lambda^{s}J^{h}v_{x}\big)~dx
                +\int_\mathbb{R} \big(\Lambda^{s}v\big)\big(\Lambda^{s}(wJ^{h}u_{x})\big)~dx,
\end{eqnarray}
where we have used (\ref{molA}) and (\ref{molB}). Since $B^{1/2}$, $\Lambda^{s}$ and $J^{h}$ are self-adjoint and commute with each other, we have
\begin{equation}
    \int_\mathbb{R} \bigg(\big(B^{1/2}\Lambda^{s}u\big)\big(B^{1/2}\Lambda^{s}J^{h}v_{x}\big)
            +\big(\Lambda^{s}v\big)\big(\Lambda^{s}J^{h}Bu_{x}\big)\bigg)~dx
    =\int_\mathbb{R} \frac{\partial}{\partial x}\big(J^{h}B^{1/2}\Lambda^{s}u\big)\big(B^{1/2}\Lambda^{s}v\big)~dx=0. \label{zeroo}
\end{equation}
If we use integration by parts for the last two integrals in (\ref{xx}),  it becomes
\begin{equation}\label{xxyeni}
    \frac{d}{dt}\mathcal{E}_{s}^{2}(t)
            =\frac{1}{2}\int_\mathbb{R} w_{t}\big( \Lambda ^{s}u\big)^{2}~dx
             -\int_\mathbb{R} w_{x}\big(\Lambda^{s}u\big)\big(\Lambda^{s}J^{h}v\big)~dx
  +\int_\mathbb{R}\big(\Lambda^{s}v\big)\bigg(\Lambda^{s}\big(wJ^{h}u_{x}\big)-J^{h}\big(w\Lambda^{s}u_{x}\big)\bigg)~dx,
\end{equation}
where we have used (\ref{zeroo}). Regarding the first two integrals on the right-hand side of  (\ref{xxyeni}) we have the  following two inequalities respectively:
\begin{eqnarray}
  &&   \int_\mathbb{R} w_{t}\big(\Lambda^{s}u\big)^{2}~dx\leq \Vert w_{t} \Vert_{L^{\infty}}~ \Vert u\Vert_{H^{s}}, \label{yyfirst} \\
  && \int_\mathbb{R} w_{x}\big(\Lambda^{s}u\big)\big(\Lambda^{s}J^{h}v\big)~dx \leq \Vert w_{x} \Vert_{L^{\infty}}~ \Vert u\Vert_{H^{s}}~\Vert v\Vert_{H^{s}}. \label{yysecond}
\end{eqnarray}
To get a similar estimate for the last integral in (\ref{xxyeni}) we make use of the commutator estimates in Lemmas \ref{lem2.1} and \ref{lem2.2}. Using $[\Lambda^{s},f]g=\Lambda^{s}(fg)-f\Lambda^{s}g$, a part of the integrand in the last integral can be written as
\begin{equation}
    \Lambda^{s}\big(wJ^{h}u_{x}\big)-J^{h}\big(w\Lambda^{s}u_{x}\big)=[\Lambda^{s},w]J^{h}u_{x}-[J^{h},w]\Lambda^{s}u_{x}.
    \label{zzzz}
\end{equation}
Regarding the first term on the right-hand side of (\ref{zzzz}), by Lemma \ref{lem2.1} we have
  \begin{equation*}
        \big\Vert [ \Lambda^{s},w] J^{h}u_{x}\big\Vert_{L^{2}}
        \leq C \big(\Vert w_{x}\Vert_{L^{\infty}}~\Vert J^{h}u_{x}\Vert_{H^{s-1}}+\Vert w\Vert_{H^{s}}~\Vert J^{h}u_{x}\Vert_{L^{\infty }}\big)
    \end{equation*}%
or, when $s>3/2$,
    \begin{equation}
        \big\Vert [ \Lambda^{s},w] J^{h}u_{x}\big\Vert_{L^{2}}
        \leq C ~\Vert w\Vert_{H^{s}}~\Vert u\Vert_{H^{s}}.  \label{yythird}
    \end{equation}%
Now, regarding the second term on the right-hand side of (\ref{zzzz}) we will use  the estimate (\ref{LAN-est}) with $r=0$.  From (\ref{LAN-est}) we have
\begin{equation}
        \big\Vert [ J^{h}, w] \Lambda^{s}u_{x}\big\Vert_{L^{2}}
        \leq C~ \Vert w_{x}\Vert_{H^{t_{0}}}~\Vert \Lambda^{s}u_{x}\Vert_{H^{-1}}
        \leq C~ \Vert w\Vert_{H^{t_{0}+1}}~\Vert \Lambda^{s}u\Vert_{L^{2}}.
    \end{equation}%
If we take $s=t_{0}+1>3/2$, we get
\begin{equation}
        \big\Vert [ J^{h}, w] \Lambda^{s}u_{x}\big\Vert_{L^{2}}
                \leq C~ \Vert w\Vert_{H^{s}}~\Vert u\Vert_{H^{s}}. \label{yyfourth}
    \end{equation}%
So, using the results obtained in (\ref{yyfirst}), (\ref{yysecond}), (\ref{yythird}) and (\ref{yyfourth}), for $s>3/2$ we get
\begin{equation}
    \frac{d}{dt}\mathcal{E}_{s}^{2}\leq C \Vert w\Vert_{Y^{s}}\mathcal{E}_{s}^{2}, \label{linest}
\end{equation}
from (\ref{xxyeni}).  Gronwall's inequality yields the result  (\ref{linestimate}).
\end{proof}
In the remaining parts of this study we require some extensions of Lemma \ref{lem2.3}. For this purpose we now  state the following two remarks regarding the systems associated with (\ref{molA})-(\ref{molB})
\begin{remark}\label{rem2.4}
    For the nonhomogeneous system
    \begin{eqnarray}
        u_{t} &=&J^{h}v_{x}+F_{1},~\text{\ \ \ \ \ \ \ \ \ \ \ \ \ \ \ \ \ \ \ }u(x,0)=u_{0}(x),  \label{nonA} \\
        v_{t} &=&BJ^{h}u_{x}+wJ^{h}u_{x}+F_{2},\text{ \ \ \ \ \ }v(x,0)=v_{0}(x),  \label{nonB}
    \end{eqnarray}
    the estimate (\ref{linest}) of Lemma \ref{lem2.3} becomes
    \begin{equation}
    \frac{d}{dt}\mathcal{E}_{s}^{2}\leq C \Vert w\Vert_{Y^{s}}\mathcal{E}_{s}^{2}+C \big(\Vert F_{1}\Vert_{H^{s}}+\Vert F_{2}\Vert_{H^{s}}\big)\mathcal{E}_{s}  \label{nonest}
\end{equation}
with $s>3/2$.
\end{remark}
\begin{remark}\label{rem2.5}
    The conclusions of  Lemma \ref{lem2.3} and Remark \ref{rem2.4} also hold when $J^{h}$ is replaced by the identity operator $I$.
\end{remark}

We now proceed with the proof of the local well-posedness of (\ref{linsysa})-(\ref{linsysb}). Since  $\Vert u_{h}(t)\Vert^{2}_{H^{s}}+\Vert v_{h}(t)\Vert^{2}_{H^{s}}\approx \mathcal{E}_{s}^{2}(t)$ for solution $(u_{h},v_{h})$ of  (\ref{molA})-(\ref{molB}), Lemma \ref{lem2.3} shows that $u_{h}$ and $v_{h}$ are bounded in $X^{s}$ and thus there is a subsequence $(u_{h_{k}},v_{h_{k}})$ weakly converging to some $\bar{u},\bar{v}\in X^{s}$  as $h_{k} \rightarrow 0$. To show the strong convergence we introduce the differences $p=u_{h_{k}}-u_{h_{m}}$ and $q=v_{h_{k}}-v_{h_{m}}$. Then, from (\ref{molA})-(\ref{molB}) we get
\begin{eqnarray}
        p_{t} &=&J^{h_{k}}q_{x}+F_{1},~\text{\ \ \ \ \ \ \ \ \ \ \ \ \ \ \ \ \ \ \ \ \ }p(x,0)=0,  \label{nonpqA} \\
        q_{t} &=&BJ^{h_{k}}p_{x}+wJ^{h_{k}}p_{x}+F_{2},\text{ \ \  \ \ \ }q(x,0)=0,  \label{nonpqB}
\end{eqnarray}
where
\begin{eqnarray}
        F_{1} &=&(J^{h_{k}}-J^{h_{m}})(v_{h_{m}})_{x}  \label{f1} \\
        F_{2} &=&B(J^{h_{k}}-J^{h_{m}})(u_{h_{m}})_{x}+w(J^{h_{k}}-J^{h_{m}})(u_{h_{m}})_{x}.  \label{f2}
\end{eqnarray}
Note that (\ref{nonpqA})-(\ref{nonpqB}) is of the form of the nonhomogeneous system (\ref{nonA})-(\ref{nonB}). If we replace $(u,v)$ in (\ref{energy}) by $(p,q)$ we get the energy  $\mathcal{E}_{s}=\mathcal{E}_{s}(p,q)$ associated with (\ref{nonpqA})-(\ref{nonpqB}).  By following Remark \ref{rem2.4},  we write the estimate
\begin{equation}
    \frac{d}{dt}\mathcal{E}_{s}^{2}\leq C \mathcal{E}_{s}^{2}+C\Big(\Vert F_{1}\Vert_{H^{s}}+\Vert F_{2}\Vert_{H^{s}}\Big)\mathcal{E}_{s}, \label{lin-est}
\end{equation}
where the term $\Vert w\Vert_{Y^{s}}$ has been incorporated into the constant. By the mollifier estimate (\ref{mol-est}) we have
\begin{displaymath}
    \big\Vert (J^{h_{k}}-J^{h_{m}})z_{x}\big\Vert_{H^{s-2}}
            \leq C ~\vert h_{k}-h_{m}\vert ~\Vert z_{x}\Vert_{H^{s-1}} \leq C ~\vert h_{k}-h_{m}\vert ~\Vert z\Vert_{H^{s}}.
\end{displaymath}
Then, one has $\Vert F_{i} \Vert_{H^{s-2}}\leq C \vert h_{k}-h_{m}\vert $ for $i=1,2$. Replacing $s$ in (\ref{lin-est})  by $s-2>3/2$, we get
\begin{equation}
    \frac{d}{dt}\mathcal{E}_{s-2}^{2}\leq C ~\mathcal{E}_{s-2}^{2}+C~\vert h_{k}-h_{m}\vert ~ \mathcal{E}_{s-2},~~~~\mathcal{E}_{s-2}^2(0)=0.
\end{equation}
Then, Gronwall's inequality gives $\mathcal{E}_{s-2}\leq C\vert h_{k}-h_{m}\vert$. Hence $u_{h_{k}},v_{h_{k}}$ are Cauchy in $X^{s-2}$; and thus they converge  in $X^{s-2}$. By uniqueness of the limit, this limit  must be the weak limit $\bar{u},\bar{v}\in X^{s}$. Finally it is quite straightforward to see that $\bar{u},\bar{v}$ indeed solve the Cauchy problem (\ref{linsysa})-(\ref{linsysb}). So we have established the local well-posedness of solutions to  (\ref{linsysa})-(\ref{linsysb}).
\begin{lemma}\label{lem2.6}
 Suppose $\mu$ satisfies  $0<c_{1}\leq  \widehat{\mu }(\xi )\leq c_{2}$ for some constants $c_{1}$ and $c_{2}$.     Let $s>7/2$  and $Y^{s}=  C\big( [0, T], H^{s}\big) \cap C^{1}\big( [0,T],H^{s-1}\big)$.  Let $u_{0},v_{0}\in H^{s}$ and $w\in Y^{s}$ with $~0<d_{1}\leq c_{1}+w(x,t)~$   for all $(x,t) \in \mathbb{R}\times [0, T] $ for some constant $d_{1}$. Then there exist unique $u,v\in Y^{s}$ satisfying (\ref{linsysa})-(\ref{linsysb}) on $\mathbb{R}\times [ 0,T]$.
\end{lemma}

\section{Local well-posedness for the nonlinear system}
\label{Sec:3}

Once having proved the well-posedness of the linearized system, the next stage is to prove the local well-posedness of the nonlinear system (\ref{systema})-(\ref{systemb}).  In the proof of the main theorem, we will make use of the nonlinear estimates (see \cite{Alinhac2007, Constantin2002}) in the following lemma:
\begin{lemma}\label{lem3.1}
    Let $h\in C^{\infty }(\mathbb{R})$ with $h(0)=0$. Then, for any $s\geq 0$ and $u,v\in L^{\infty}\cap H^{s}$,
    \begin{enumerate}
    \item $h(u)\in H^{s}$ with $\ \Vert h(u)\Vert _{H^{s}}\leq C_{1}\Vert u\Vert _{H^{s}}$ where $C_{1}$ depends on $h$ and $\Vert u\Vert _{L^{\infty }}.$
    \item $\Vert h(u)-h(v)\Vert _{H^{s}}\leq C_{2}\Vert u-v\Vert _{H^{s}}$ where $C_{2}$ depends on $h$ and $\Vert u\Vert_{L^{\infty }},\Vert v\Vert_{L^{\infty}},\Vert u\Vert_{H^{s}},$ and $\Vert v\Vert _{H^{s}}.$
    \end{enumerate}
\end{lemma}
The main result of this section is:
\begin{theorem}\label{theo3.2}
    Suppose $\mu$ satisfies  $0<c_{1}\leq  \widehat{\mu}(\xi )\leq c_{2}$ for some constants $c_{1}$ and $c_{2}$.     Let $s>7/2$ and $u_{0},v_{0}\in H^{s}$ be sufficiently small. Then there exists some $T>0$ so that the nonlinear system (\ref{systema})-(\ref{systemb}) is locally well posed with solution $u, v\in Y^{s}=  C\big( [0, T], H^{s}\big) \cap C^{1}\big( [0,T],H^{s-1}\big)$.
\end{theorem}
\begin{proof}
    The proof is done via Picard's iterations. We employ the energy estimate in Lemma \ref{lem2.3} taking $J^{h}=I$ and $w=g^{\prime}(u)$ in (\ref{molA})-(\ref{molB}). Due to (\ref{hyperbolic}) we ask the initial value $u(x,0)=u_{0}(x)$ to satisfy
    \begin{equation*}
        0<d_{1}\leq c_{1}+g^{\prime }(u_{0}(x))
    \end{equation*}%
    for all $x\in \mathbb{R}$.  Since $g^{\prime}(0)=0,$  we have $0<c_{1}\leq 1+ g^{\prime }(z)$ for sufficiently small $\vert z\vert$. By the Sobolev embedding theorem,  there is some $\gamma $ so that $0<d_{1}\leq c_{1}+g^{\prime}(z(x,t))$ whenever $\Vert z(t)\Vert _{H^{s}}\leq \gamma$.

    We assume that $\Vert u_{0}\Vert _{H^{s}}+\Vert v_{0}\Vert_{H^{s}}\leq \frac{\gamma }{2}$. Consequently  $w_{0}=g^{\prime }(u_{0})$ satisfies (\ref{hyperbolic}). We now consider  the iterates  $(u^{n+1},v^{n+1})$ solving  the following linear system
    \begin{eqnarray*}
        u_{t}^{n+1} &=&v_{x}^{n+1},~\text{ \ \ \ \ \ \ \ \ \ \ \ \ \ \ \ \ \ }u^{n+1}(x,0)=u_{0}(x), \\
        v_{t}^{n+1} &=&Bu_{x}^{n+1}+w_{n}u_{x}^{n+1},\text{ \ }v^{n+1}(x,0)=v_{0}(x)
    \end{eqnarray*}%
    with $(u^{0},v^{0})=(u_{0},v_{0})$ and $w_{n}=g^{\prime}(u^{n})$. By the energy estimate (\ref{linestimate}) of Lemma \ref{lem2.3}, we have
    \begin{equation*}
        \Vert u^{1}(t)\Vert _{H^{s}}+\Vert v^{1}(t)\Vert _{H^{s}}
        \leq \big(\Vert u_{0}\Vert _{H^{s}}+\Vert v_{0}\Vert _{H^{s}}\big)e^{Ct}\leq \frac{\gamma}{2}e^{Ct}\leq \gamma
    \end{equation*}%
    for   $t\leq T_{0}=\frac{\log 2}{C}$  with $C=C\big(\Vert g^{\prime }(u_{0})\Vert _{H^{s}}\big)\leq C(\gamma)$. For all $(u^{n},v^{n})$ a similar estimate holds. So $w_{n}$ satisfies (\ref{hyperbolic})  for all $x$ and $t\leq T_{0}=\frac{\log 2}{C}$. We now estimate the differences $( p^{n+1}, q^{n+1}) =(u^{n+1}-u^{n},v^{n+1}-v^{n})$. They satisfy the following system
    \begin{eqnarray}
        p_{t}^{n+1} &=&q_{x}^{n+1},
        ~\text{\ \ \ \ \ \ \ \ \ \ \ \ \ \ \ \ \ \ \ \ \ \ \ \ \ \ \ \ \ \ \ \ \ \ \ \ \ \ \ \ \ \ \ \ \ \ \ \ \ \ \ \ \ \ \ \ \ }p^{n+1}(x,0)=0, \label{ppp}\\
        q_{t}^{n+1} &=&Bp_{x}^{n+1}+g^{\prime}(u^{n})p_{x}^{n+1}+\big(g^{\prime}(u^{n})-g^{\prime }(u^{n-1})\big)u_{x}^{n},~\text{\ }~\text{ \ }q^{n+1}(x,0)=0. \label{qqq}
    \end{eqnarray}%
    Using the nonhomogeneous energy estimate (\ref{nonest}) of Remark \ref{rem2.4}  with  $F_{1}=0$, $F_{2}=\big(g^{\prime}(u^{n})-g^{\prime }(u^{n-1})\big)u_{x}^{n}$, $\mathcal{E}_{s}=\mathcal{E}_{s}(p,q)$ and $\mathcal{E}_{s}(0)=0$, we obtain
   \begin{equation}
        \mathcal{E}_{s}(t)\leq C ~\Vert F_{2}\Vert_{X^{s}}~\big(e^{C t}-1\big)  \label{nonestimate}
    \end{equation}
    for $t\in [0,T]$. Replacing $s$ by $s-1$ and using the definitions for $\mathcal{E}_{s-1}$ and $F_{2}$ we obtain
    \begin{eqnarray*}
        \Vert p^{n+1}(t)\Vert _{H^{s-1}}+\Vert q^{n+1}(t)\Vert_{H^{s-1}}
        &\leq & C ~\Big\Vert \Big(g^{\prime }(u^{n}(t))-g^{\prime}(u^{n-1}(t))\Big)u_{x}^{n}(t)\Big\Vert_{H^{s-1}}~(e^{Ct}-1) \\
        &\leq & C ~\big\Vert u^{n}(t)-u^{n-1}(t)\big\Vert_{H^{s-1}}~\big\Vert u_{x}^{n}(t)\big\Vert_{H^{s-1}}~(e^{Ct}-1) \\
        &\leq & C ~\big\Vert p^{n}(t)\big\Vert_{H^{s-1}}~(e^{Ct}-1),
    \end{eqnarray*}%
    where we have also used Lemma \ref{lem3.1}.     Choosing $T<T_{0}$ now so that $e^{CT}-1\leq \frac{1}{2C}$ we see that for $ t\leq T$,
    \begin{equation*}
        \Vert p^{n+1}(t)\Vert _{H^{s-1}}+\Vert q^{n+1}(t)\Vert_{H^{s-1}}
        \leq \frac{1}{2}\big(\Vert p^{n}(t)\Vert _{H^{s-1}}+\Vert q^{n}(t)\Vert _{H^{s-1}}\big)\leq \cdots \leq \frac{C}{2^{n}}.
    \end{equation*}%
    This shows that $(u^{n}, v^{n})$ forms a Cauchy sequence in $H^{s-1}$ and the limit $( u,v) $ will be a solution in $H^{s-1}.$ Finally, considering the weak limit of  $(u^{n},v^{n})$ in $X^{s}$ as was done in the proof of Lemma \ref{lem2.6}, we obtain regularity, namely that $u,v\in C\big([ 0,T], H^{s}\big)$.
\end{proof}

\section{The vanishing dispersion limit of the nonlocal equation}
\label{Sec:4}

In this section we will consider a parameterized form of (\ref{nw}) in which the operator $B$ is replaced by the family of convolution operators  $B_{\varepsilon}$. We then show that for two different types of the vanishing dispersion limit, solutions of the Cauchy problem for the parameterized form of  (\ref{nw}) converge to the corresponding solution of the classical elasticity equation (\ref{classical}). Before we start working with two different forms of the kernel, it is important to remember that the energy estimate of Lemma \ref{lem2.3} requires $s>3/2$ while the local well-posedness result of Theorem \ref{theo3.2} requires $s>7/2$.

\subsection{First type of vanishing dispersion limit}
In this case we assume that $B_{\varepsilon}=(\delta +\varepsilon \beta)\ast$ with a small parameter $\varepsilon$, a fixed $L^{1}$  function $\beta$ and the Dirac measure $\delta$, in which  $B_{\varepsilon}u=u+\varepsilon (\beta\ast u)$. In other words, we consider the nonlocal equation  (\ref{nwplus}) and study convergence of the solutions to the nonlocal equation to the corresponding solution of (\ref{classical}) as $\varepsilon$ approaches zero. In terms of the first-order nonlinear systems this means that we are comparing solutions $(u^{\varepsilon},v^{\varepsilon})$
of the Cauchy problem
\begin{eqnarray}
    && u_{t}=v_{x},\text{ \ \ \ \ \ \ \ \ \ \ \ \ \ \ \ \ \ \ \ \ \ \ \ \ \ \ \ \ \ \ \ }u(x,0)=u_{0}(x),~~\label{firsta} \\
    && v_{t}=\varepsilon (\beta \ast u_{x})+u_{x}+(g(u))_{x},\text{ \ \ \ }v(x,0)=v_{0}(x)  \label{firstb}
\end{eqnarray}%
with the  solution $(u,v)$ of the $p$-system (which corresponds to the case   $\varepsilon=0$) for the same initial data.   By the local existence theorem we know that for $u_{0}, v_{0}\in  H^s$  sufficiently small and for some $T > 0$ both nonlinear systems are locally well-posed with solutions in $Y^{s}$. Moreover,  with a careful examination of the involved energies, the existence time can be chosen independent of  $\varepsilon \geq 0$.

Then the differences $(p,q)=(u^{\varepsilon}-u, v^{\varepsilon}-v)$ satisfy
\begin{eqnarray}
    p_{t} &=&q_{x},
    ~\text{\ \ \ \ \ \ \ \ \ \ \ \ \ \ \ \ \ \ \ \ \ \ \ \ \ \ \ \ \ \ \ \ \ \ \ \ \ \ \ \ \ \ \ }p(x,0)=0,
    \label{typ1a}\\
    q_{t} &=&\varepsilon (\beta\ast u^{\varepsilon}_{x})+p_{x}+\big(g(u^{\varepsilon})-g(u)\big)_{x},~\text{\ }~\text{ \ }q(x,0)=0.  \label{type1b}
\end{eqnarray}%
We now  rearrange the last term in (\ref{type1b}) as follows
\begin{eqnarray*}
    \big(g(u^{\varepsilon})-g(u)\big)_{x}
     &=&g^{\prime}(u^{\varepsilon})u^{\varepsilon}_{x}-g^{\prime}(u)u_{x}
            -g^{\prime}(u^{\varepsilon})u_{x}+g^{\prime}(u^{\varepsilon})u_{x},  \\
     &=& g^{\prime}(u^{\varepsilon})p_{x}+\big(g^{\prime}(u^{\varepsilon})-g^{\prime}(u)\big)u_{x}.
\end{eqnarray*}%
Using this result in (\ref{type1b}) we get
\begin{eqnarray}
    p_{t} &=&q_{x},
    ~\text{\ \ \ \ \ \ \ \ \ \ \ \ \ \ \ \ \ \ \ \ \ \ \ \ \ \ \ \ \ \ \ \ \ \ \ \ \ \ \ \ \  \ \ \ \ \ \ \ \  \ \  \ \ \ \  \ \ \ \ \ \ \ \ \ \  \ \ \ \ \ \ \ \ \ \ \ }p(x,0)=0,
    \label{typ1ay}\\
    q_{t} &=&\varepsilon (\beta\ast p_{x})+p_{x}+ g^{\prime}(u^{\varepsilon})p_{x} +\varepsilon (\beta \ast u_{x})+\big(g^{\prime}(u^{\varepsilon})-g^{\prime}(u)\big)u_{x},~\text{\ }~\text{ \ }q(x,0)=0.  \label{type1by}
\end{eqnarray}%
This system is of the form (\ref{nonA})-(\ref{nonB}) with $J^{h}=I$, $Bp_{x}=\varepsilon (\beta\ast p_{x})+p_{x}$, $w=g^{\prime}(u^{\varepsilon})$, $F_{1}=0$ and $F_{2}=\varepsilon (\beta\ast u_{x})+\big(g^{\prime}(u^{\varepsilon})-g^{\prime}(u)\big)u_{x}$. Thus, noting that $\mathcal{E}_{s}^{2}(t)\approx \Vert p(t)\Vert^{2}_{H^{s}}+\Vert q(t)\Vert^{2}_{H^{s}}$, we get
\begin{equation}
    \frac{d}{dt}\mathcal{E}_{s}^{2}\leq C ~\Vert g^{\prime}(u^{\varepsilon})\Vert_{Y^{s}}~\mathcal{E}_{s}^{2}+C~\Vert F_{2}\Vert_{H^{s}}~\mathcal{E}_{s}  \label{typ1est}
\end{equation}
from (\ref{nonest}). On the other hand, we have
\begin{eqnarray}
  \Vert F_{2}\Vert_{H^{s-1}}
    &\leq &   \varepsilon \Vert (\beta \ast u_{x})\Vert_{H^{s-1}}
            +\Vert\big(g^{\prime}(u^{\varepsilon})-g^{\prime}(u)\big)u_{x}\Vert_{H^{s-1}} \\
    &\leq &  \varepsilon C \Vert u\Vert_{H^{s}}
            +C \Vert u^{\varepsilon}-u\Vert_{H^{s-1}~}~ \Vert u\Vert_{H^{s}} \\
    &\leq & C \big(\varepsilon+\Vert p\Vert_{H^{s-1}}\big)\leq C \big(\varepsilon+\mathcal{E}_{s-1}\big).
\end{eqnarray}
Replacing $s$ in (\ref{typ1est}) by $s-1$,  this in turn yields
\begin{equation*}
        \frac{d}{dt}\mathcal{E}^{2}_{s-1}
            \leq C~\mathcal{E}^{2}_{s-1}+\varepsilon C ~\mathcal{E}_{s-1}, ~~~~\mathcal{E}_{s-1}(0)=0,
    \end{equation*}%
from which  we get $\mathcal{E}_{s-1}(t)\leq \varepsilon \big(e^{Ct}-1\big)$; namely
\begin{theorem}\label{theo4.1}
    Let $s>7/2$ and  $u_{0},v_{0}\in H^{s}$ be sufficiently small. Suppose $0<c_{1}\leq  1+\varepsilon \widehat{\beta }\leq c_{2} $ for all sufficiently small $\varepsilon$ and for some constants $c_{1}$, $c_{2}$.  Let $(u^{\varepsilon},v^{\varepsilon})$ and $(u,v)$ be solutions of the Cauchy problem  (\ref{firsta})-(\ref{firstb}) on $[0,T]$  corresponding  to the cases  $\varepsilon>0$ and $\varepsilon=0$, respectively. Then we have the estimate
        \begin{equation*}
             \big\Vert u^{\varepsilon}(t) -u(t) \big\Vert_{H^{s-1}}+\big\Vert v^{\varepsilon}(t) -v(t) \big\Vert _{H^{s-1}}\leq \varepsilon \big(e^{Ct}-1\big) ~~~~\text{for all} ~~ t\in [0,T].
        \end{equation*}
\end{theorem}

\subsection{Second type of vanishing dispersion limit}
In this case we assume that $B_{\varepsilon}u=\mu_{\varepsilon}\ast u$ with $\mu_{\varepsilon}(x)=\frac{1}{\sqrt{\varepsilon}}\mu(\frac{x}{\sqrt{\varepsilon}})$ and a small parameter $\varepsilon$. Here $\mu$ is an even finite Borel measue with  $\int_{\mathbb{R}} ~d\mu=1$. As $\varepsilon$ tends to zero, $\mu_{\varepsilon}$ converges to the Dirac measure $\delta$. The Fourier transforms of $\mu_{\varepsilon}$ and $\mu$ satisfy $\widehat{\mu_{\varepsilon}}(\xi)=\widehat{\mu}(\sqrt{\varepsilon} \xi)$.   We also assume the second moment condition  $\int_{\mathbb{R}} x^{2}d\vert \mu\vert <\infty$, so that $\widehat{\mu}\in C^2$, $\widehat{\mu}(0)=1$, $\widehat{\mu}^{\prime}(0)=0$ and the second derivative $\widehat{\mu}^{\prime\prime}$ is bounded. The Taylor expansion around $\xi=0$ gives
\begin{equation}
        \widehat{\mu}(\xi)=1+\frac{1}{2}\widehat{\mu}^{\prime\prime}(c)\xi^2   \label{bhexp}
\end{equation}
for some $c\in \mathbb{R}$. As $\widehat{\mu}^{\prime\prime}$ is bounded due to the moment condition; we have $\big\vert \widehat{\mu}(\xi)-1\big\vert \leq C\xi^2$ with $C=\frac{1}{2}\sup_{c\in \mathbb{R}}\big\vert \widehat{\mu}^{\prime\prime}(c)\big\vert $. Then from the inequality
\begin{equation}
       \big\vert \widehat{\mu}_{\varepsilon}(\xi)-1\big\vert =\big\vert \widehat{\mu}(\sqrt{\varepsilon}\xi)-1\big\vert \leq C\varepsilon \xi^2,  \label{behexp}
\end{equation}
we get the estimate
\begin{equation*}
        \Vert \mu_{\varepsilon}\ast u-u\Vert_{H^{s-2}}\leq \varepsilon \Vert u\Vert_{H^{s}}.
\end{equation*}
Notice that $B_{\varepsilon}$ converges to $I$ as $\varepsilon \rightarrow 0$ and we get  (\ref{classical}). Again, we study convergence of the solutions of the nonlocal equation $u_{tt}=\mu_{\varepsilon}\ast u_{xx}+(g(u))_{xx}$ to the corresponding solution of (\ref{classical}) as $\varepsilon$ goes to zero. In terms of the first-order nonlinear systems this means that we are comparing solutions $(u^{\varepsilon},v^{\varepsilon})$
of the Cauchy problem
\begin{eqnarray}
    && u_{t}=v_{x},\text{ \ \ \ \ \ \ \ \ \ \ \ \ \ \  \ \ \ \ \ \ \ \ }u(x,0)=u_{0}(x),~~\label{seconda} \\
    && v_{t}=\mu_{\varepsilon}\ast u_{x}+(g(u))_{x},\text{ \ \ \ }v(x,0)=v_{0}(x)  \label{secondb}
\end{eqnarray}%
with the solution $(u,v)$ of the same Cauchy problem  when $\mu_{\varepsilon}$ is the Dirac delta function.   By the local existence theorem we know that for $u_{0}, v_{0}\in  H^s$  sufficiently small and for some $T > 0$ both nonlinear systems are locally well-posed with solutions in $Y^s$. As in the previous subsection,  with a careful examination of the involved energies, the existence time can be chosen independent of  $\varepsilon \geq 0$.

Then the differences $(p,q)=(u^{\varepsilon}-u, v^{\varepsilon}-v)$ satisfy
\begin{eqnarray*}
    p_{t} &=&q_{x},
    ~\text{\ \ \ \ \ \ \ \ \ \ \ \ \ \ \ \ \ \ \ \ \ \ \ \ \ \ \ \ \ \ \ \ \ \ \ \ \ \ \ \ \ \ \ \ \ \ \ \ \ \ \ \ \ \ \ \ \ \ \ \ \ \ \ \ \ \ \ \ \ \ \ \ \ \  }p(x,0)=0, \\
    q_{t} &=&\mu_{\varepsilon}\ast p_{x}+g^{\prime}(u^{\varepsilon})p_{x}+(\mu_{\varepsilon}\ast u_{x}-u_{x})+\big(g^{\prime}(u^{\varepsilon})-g^{\prime}(u)\big)u_{x},~\text{\ }~\text{ \ }q(x,0)=0.
\end{eqnarray*}%
The system is of the form (\ref{nonA})-(\ref{nonB}) with $J^{h}=I$, $Bp_{x}=\mu_{\varepsilon}\ast p_{x}$, $w=g^{\prime}(u^{\varepsilon})$, $F_{1}=0$ and $F_{2}=(\mu_{\varepsilon}\ast u_{x}-u_{x})+\big(g^{\prime}(u^{\varepsilon})-g^{\prime}(u)\big)u_{x}$. Again, (\ref{nonest}) reduces to (\ref{typ1est}) but with different $F_{2}$. We have the estimate
\begin{eqnarray}
  \Vert F_{2}\Vert_{H^{s-3}}
    &\leq &    \Vert \mu_{\varepsilon}\ast u_{x}-u_{x}\Vert_{H^{s-3}}
            +\Vert\big(g^{\prime}(u^{\varepsilon})-g^{\prime}(u)\big)u_{x}\Vert_{H^{s-3}} \\
    &\leq &  \varepsilon C \Vert u\Vert_{H^{s}}
            +C \Vert u^{\varepsilon}-u\Vert_{H^{s-3}~}~ \Vert u\Vert_{H^{s-2}} \\
    &\leq & C \big(\varepsilon+\Vert p\Vert_{H^{s-3}}\big)\leq C \big(\varepsilon+\mathcal{E}_{s-3}\big).
\end{eqnarray}
With this result and $s-3>3/2$, (\ref{nonest}) takes the form
\begin{equation*}
         \frac{d}{dt}\mathcal{E}^{2}_{s-3}
            \leq C~\mathcal{E}^{2}_{s-3}+\varepsilon C ~\mathcal{E}_{s-3}, ~~~~\mathcal{E}_{s-3}(0)=0,
    \end{equation*}%
So finally we get $\mathcal{E}_{s-3}(t)\leq \varepsilon\big(e^{Ct}-1\big)$; namely
\begin{theorem}\label{theo4.2}
    Let $s>9/2$ and  $u_{0},v_{0}\in H^{s}$ be sufficiently small.  Suppose $\int_{\mathbb{R}} ~d\mu=1$,  $\int_{\mathbb{R}} x^2~d\vert\mu\vert<\infty$  and $0<c_{1}\leq  \widehat{\mu }\leq c_{2} $  for some constants $c_{1}$, $c_{2}$. Let $(u^{\varepsilon},v^{\varepsilon})$ and $(u,v)$ be solutions of the Cauchy problem  (\ref{seconda})-(\ref{secondb}) on $[0,T]$  corresponding  to the cases  $\varepsilon>0$ and $\mu_{\varepsilon}=\delta$, respectively. Then we have the estimate
        \begin{equation*}
             \big\Vert u^{\varepsilon}(t) -u(t) \big\Vert_{H^{s-3}}+\big\Vert v^{\varepsilon}(t) -v(t) \big\Vert _{H^{s-3}}\leq \varepsilon \big(e^{Ct}-1\big) ~~~~\text{for all} ~~ t\in [0,T].
        \end{equation*}
\end{theorem}
As an example of the second type of vanishing dispersion limit, we consider the measure $\mu=\frac{1}{5}(\delta_{-1}+3\delta+\delta_{1})$ with the Dirac measure $\delta$ and its shifts. Explicitly $\int f ~d\mu=\frac{1}{5}(f(-1)+3f(0)+f(1))$. Then the Fourier transform of $\mu$ is
\begin{displaymath}
    \widehat{\mu}(\xi)=\frac{1}{5}(e^{-i\xi}+3+e^{i\xi})=\frac{1}{5}(3+2\cos \xi)\geq \frac{1}{5}
\end{displaymath}
and it satisfies the positivity condition. Moreover $\int_{\mathbb{R}} d\mu=\widehat{\mu}(0)=1$. Theorem \ref{theo4.2} says that as $\varepsilon \rightarrow 0$, any solution of the equation
\begin{displaymath}
    u_{tt}(x,t)=\frac{1}{5}\big( u_{xx}(x+\sqrt{\varepsilon},t)+3u_{xx}(x,t)+ u_{xx}(x-\sqrt{\varepsilon},t)\big)+g(u(x,t))_{xx}
\end{displaymath}
approaches the corresponding solution of the classical elasticity equation  (\ref{classical}).

\end{document}